\documentclass[11pt,english]{article}
\usepackage{tikz}
\usepackage[T1]{fontenc}
\usepackage[latin9]{inputenc}
\usepackage[letterpaper]{geometry}
\usepackage{array}
\usepackage{amsthm}
\usepackage{amsmath}
\usepackage{setspace}
\usepackage{amssymb}
\usepackage{amsfonts}
\usepackage{graphicx}
\usepackage{cleveref}
\crefname{lemma}{Lemma}{lemmas}
\crefname{claim}{Claim}{claims}
\crefname{corollary}{Corollary}{corollaries}
\crefname{theorem}{Theorem}{theorems}
\crefname{fact}{Fact}{facts}
\crefname{proposition}{Proposition}{propositions}

\crefname{conjecture}{Conjecture}{Conjectures}

\makeatletter

\newtheorem{conjecture}{Conjecture}
\newtheorem{theorem}[conjecture]{Theorem}
\newtheorem{corollary}[conjecture]{Corollary}
\newtheorem{proposition}[conjecture]{Proposition}
\newtheorem{lemma}[conjecture]{Lemma}

\begin{document}

\title{Few products, many $h$-fold sums}
		
\author{Albert Bush, Ernie Croot}
\date{\today}
\maketitle

\begin{abstract}
Improving upon a technique of Croot and Hart, we show that for every $h$, there exists an $\epsilon > 0$ such that if $A \subseteq \mathbb{R}$ is sufficiently large and $|A.A| \le |A|^{1+\epsilon}$, then $|hA| \ge |A|^{\Omega(e^{\sqrt{c\log{h}}})}$. 
\end{abstract}

\section{Introduction}
For a set $A \subseteq \mathbb{R}$, the sumset, product set, $h$-fold sumset, and $h$-fold product set are defined as
\[ A+A := \{ a+a': a,a' \in A \}, \]
\[ A.A := \{ a.a' : a,a' \in A \}, \]
\[ hA := \{ a_1 + \ldots + a_h: a_i \in A \}, \]
\[ A^{(h)} := \{ a_1 \cdot \ldots \cdot a_h : a_i \in A \}. \]
We prove the following theorem:
\begin{theorem}\label{mainthm}
For any $h \in \mathbb{N}$, there exists an $\epsilon = \epsilon(h) > 0$ such that the following holds: there exists an $n_0= n_0(\epsilon, h)$ such that if $A \subseteq \mathbb{R}$ is of size $n \ge n_0$ and $|A.A| \le |A|^{1+\epsilon}$, then
\[ |h A | \ge |A|^{ce^{\sqrt{\frac{1}{100}\log{h}}}  } \]
for some absolute constant $c$.
\end{theorem}
\subsection{Background}
  In 1983, Erd\H{o}s and Szemer\'edi stated the following two conjectures \cite{erdosszemeredi}:
\begin{conjecture}\label{sumprodconj}(Sum-Product Problem) For any $\epsilon > 0$, there exists an $n_0 = n_0(\epsilon)$ such that if $A \subseteq \mathbb{R}$ is of size $n \ge n_0$, then
\[ |A.A| + |A+A| \ge |A|^{2-\epsilon} \]
\end{conjecture}
\begin{conjecture}\label{hfoldsumprodconj}($h$-fold Sum-Product Problem)
For any $\epsilon > 0$ and for any $h\in \mathbb{N}$, there exists an $n_0$ such that if $A \subseteq \mathbb{R}$ is of size $n \ge n_0$, then
\[ |hA| + |A^{(h)}| \ge |A|^{h - \epsilon} \]
\end{conjecture}
Although resolution of either conjecture is currently out of reach, there has been considerable progress on \cref{sumprodconj}.
\begin{theorem}\label{sumprodthm}(\cite{elekes},\cite{ford},\cite{solymosi1},\cite{solymosi2})
There exists an $0 < \epsilon < 1$ and an absolute constant $c > 0$ such that for any $A \subseteq \mathbb{R}$
\[ |A+A| + |A.A| \ge c|A|^{1+\epsilon} \]
\end{theorem}
Initially, results were only proven when $A \subseteq \mathbb{Z}$.  In that case \cref{sumprodthm} was first proved by Erd\H{o}s and Szemer\'edi with an unspecified, but fixed value $\epsilon$ \cite{erdosszemeredi}.  Their method was refined by Nathanson and then Chen who showed one could take $\epsilon = 1/31$ and $\epsilon = 1/5$ respectively \cite{nathanson},\cite{chen}.  In the case when one assumes $A \subseteq \mathbb{R}$, Ford proved one could take $\epsilon = 1/15$.  Elekes showed one could take $\epsilon = 1/4$ in $\mathbb{R}$ by introducing a beautiful correspondence between incidence geometry and sum-product inequalities \cite{elekes}.  Solymosi expanded upon this connection and showed one can take $\epsilon = 3/11 - \delta$ \cite{solymosi1}, and then a few years later, $\epsilon = 1/3 - \delta$ \cite{solymosi2} for any $\delta > 0$ given that $A$ is sufficiently large.

Progress on \cref{hfoldsumprodconj} has been much slower. For subsets of the integers, Bourgain and Chang showed that one can take the exponent of $|A|$ to be a function that tends to infinity along with $h$:
\begin{theorem}\label{bourgchang}\cite{bourgainchang}
For every $b > 0$, there exists and $h\in \mathbb{N}$ such that for any $A \subseteq \mathbb{Z}$
\[ |hA| + |A^{(h)}| \ge |A|^{b} \]
\end{theorem}
One can take $b$ to be on the order of $(\log{h})^{1/4}$.  Unfortunately, there have not been any successful attempts at proving a similar result to \cref{bourgchang} for real-valued sets.  A predecessor to \cref{bourgchang} was proved by Chang several years earlier.
\begin{theorem}\label{chang}\cite{chang1}
For any $h \in \mathbb{N}$, there exists a $K = K(h) > 0$ such that if $A \subseteq \mathbb{Z}$ and $|A.A| \le K|A|$, then
\[ |hA| \ge c(K,h) |A|^{h}. \]
\end{theorem}
The restriction that $|A.A| \le K|A|$ allowed Chang to apply Freiman's theorem to deduce strong multiplicative structure in $A$.  However, even with the best known bounds on Freiman's theorem one cannot take $K$ up to a power of $|A|$.  Several years later, Chang showed that one can apply the Subspace Theorem to easily deduce \cref{chang} for real-valued sets \cite{chang2}.  In that same paper, she also proved a version of \cref{chang} that avoided the use of Freiman's theorem, but her method was restricted to integral sets.  This method allows one to infer information about sets with product set equal to some power of $|A|$.
\begin{theorem}\label{chang2}\cite{chang2}
For any $h \in \mathbb{N}$, there exists an $\epsilon > 0$ such that if $A \subseteq \mathbb{Z}$ and $|A.A| \le |A|^{1+\epsilon}$, then
\[ |hA| \ge |A|^{h-\delta_h(\epsilon)} \]
where $\delta_h \rightarrow 0$ as $\epsilon \rightarrow 0$.
\end{theorem}
It would be desirable to have a version of \cref{chang2} for real-valued sets.  Croot and Hart were able to prove such a theorem but with a weaker conclusion.
\begin{theorem}\label{croothartthm1}\cite{croothart}
For every $h \in \mathbb{N}$ there exists an $\epsilon' := \epsilon'(h)$ such that the following holds: for any $0< \epsilon < \epsilon'$, there exists an $n_0 := n_0(\epsilon, h)$ such that if $A \subseteq \mathbb{R}$ is of size $n \ge n_0$ and $|A.A| \le |A|^{1+\epsilon}$, then
\[ |hA| \ge |A|^{c\log{h/2} - f_h(\epsilon)} \] 
where $c$ is an absolute constant, and $f_h(\epsilon) \rightarrow 0$ as $\epsilon \rightarrow 0$.
\end{theorem}
Similar bounds have also been shown in unpublished papers \cite{konyagin} \cite{liangpanli} , but it would be desirable for many applications to have the exponent of $A$ grow faster than logarithmically with $h$.  Croot and Hart also proved a theorem on the $h$-fold sum $h(A.A)$ by introducing a method that used bounds on the Tarry-Escott problem.  
\begin{theorem}\label{croothartthm2}\cite{croothart}
For every $h \in \mathbb{N}$ there exists an $\epsilon = \epsilon(h) > 0$ such that the following holds: there exists an $n_0 := n_0(\epsilon, h)$ such that if $A \subseteq \mathbb{R}$ is of size $n \ge n_0$ then
\[ |h(A.A)| \ge |A|^{\Omega((h/\log{h})^{1/3})}. \]
\end{theorem}
The goal of this paper -- that is, in proving \cref{mainthm} -- is to greatly extend the techniques used in proving \cref{croothartthm2} to give a much stronger bound on $hA$ than is found in \cref{croothartthm1}.

\subsection{Layout and Notation.} In Section 2, we list some well-known additive combinatorial results that we will need.  We also include several lemmas that are directly from \cite{croothart}.  For completeness, we include the proofs of these lemmas.  In Section 3 and 4, we prove new, key lemmas that we will need to prove \cref{mainthm}.  Section 5 contains the proof of \cref{mainthm}.  In addition to the notation introduced in the beginning, we define the difference and quotient set as follows:
\[ A-B := \{ a-b : a\in A, b\in B \} \]
\[ A/B := \{a/b: a\in A, b\in B \} \] 
All sets are assumed to be finite subsets of $\mathbb{R}$ unless indicated otherwise.  The additive energy $E(A,B)$ is defined as
\[ \{ (a,b, a', b') \in A\times B \times A \times B: a+b = a' + b' \}. \]
We say that $f \gg g$ if $g = O(f)$ and $f \gg_k g$ if $f(n) \ge c(k)g(n)$ for $n$ sufficiently large.  We say that a polynomial $p(x)$ vanishes at $x=a$ to order $j$ if $x=a$ is a root of order $j$ but not $j+1$.  All graphs are finite and undirected.  For a graph $(G,E)$, $\Delta(G)$ denotes the maximum degree of $G$.  We will abuse notation and denote $|G|$ as $|V(G)|$.  
\section{Lemmas and Known Results}
%
%
The Pl\"unnecke-Ruzsa inequality is ubiquitious in additive combinatorics and will be needed in our proof.
\begin{lemma}[Pl\"unnecke-Ruzsa Inequality]\cite{ruzsaplunnecke}\cite{taovu}
Let $A$ be a subset of a finite abelian group such that $|A+A| \le c|A|$.  Then, $|kA - \ell A| \le c^{k+\ell} |A|$.
\end{lemma}
We will also need the following lemma which exists in many different forms (\cite{taovu}, Chap. 2).
%
%
\begin{lemma}\label{easyint}
Let $X, Y \subseteq \mathbb{R}$.  Then,
\[ |X+Y| \ge \frac{|X||Y|}{|(X-X) \cap (Y-Y)|}. \]
In particular, if $(X-X) \cap (Y-Y) = \{0\}$, then $|X+Y| = |X||Y|$.
\end{lemma}
\begin{proof}
The additive energy of $X$ and $Y$ can be bounded from above by
\begin{align*}E(X,Y) & := |\{(x,x',y,y') \in X \times X \times Y \times Y: x + y = x' + y' \}| = |\{(x,x',y,y'): x-x' = y-y' \}|
\\ & = \sum_{t \in X-X \cap Y-Y} |\{(x,x', y, y'): x-x' = t = y-y' \}| \le |(X-X) \cap (Y-Y)||X||Y|
\end{align*}
On the other hand, one can use Cauchy-Schwarz to bound the additive energy from below:
\[ E(X,Y) = \sum_{s \in X+Y} |\{(x,y) \in X \times Y: x+y = s \}|^2 \ge \frac{|X|^2|Y|^2}{|X+Y|}. \]
Combining the two inequalities proves the lemma.
\end{proof}
We will use several lemmas from \cite{croothart} whose proofs we include for completeness.  First, we state a result of Wooley on the Tarry-Escott problem \cite{wooley}.
%
%
\begin{theorem}\label{wooleylemma}
For every $k \ge 3$, there exists two distinct sets
\[ \{ a_1 , \ldots , a_s \}, \{ b_1, \ldots , b_s \} \subseteq \mathbb{Z} \]
such that for all $j = 1, \ldots , k$
\[ \sum_{i=1}^s a_i^j = \sum_{i=1}^s b_i^j  \]
but 
\[ \sum_{i=1}^s a_i^{k+1} \neq \sum_{i=1}^s b_i^{k+1}. \]
Moreover, $s < (5/8)(k+1)^2$.
\end{theorem}
We will need a useful corollary of this result.
%
%
\begin{corollary}\label{wooleycor}
For all $k \ge 2$, there exists a monic polynomial $f(x)$ having coefficients only $0, 1, -1$ having at most $k^2$
nonzero terms such that $f(x)$ vanishes at $x=1$ to order exactly $k$.
\end{corollary}
\begin{proof}
For $k=2, 3$, verify the corollary by hand by considering $(x-1)(x^2-1)$ and $(x-1)(x^2-1)(x^4-1)$.  For $k \ge 4$, we use \cref{wooleylemma}.  Note that for $k \ge 4$, we have that $k^2 \ge (5/8)(k+1)^2$.  Apply Lemma \ref{wooleylemma} to get two distinct sets $\{ a_1, \ldots , a_s\}$ and $\{b_1, \ldots , b_s\}$ with the properties stated in the lemma.  If these sets are not in $\mathbb{Z}_{\ge 0}$, then let $a := \min \{a_1, \ldots , a_s, b_1, \ldots , b_s \}$ otherwise, $a:=0$.  Let
\[ f(x) := x^{-a} \sum_{i=1}^s x^{a_i} - x^{b_i}. \]
Since the sets are distinct, it is clear that the polynomial is monic, has at most $k^2$ nonzero terms, and only has coefficients $1$, and $-1$.  To see that $f$ has the correct order of vanishing at $x = 1$, we use the fact that $f$ vanishes at $x=1$ to order exactly $k$ if and only if its first $k$ derivatives vanish at $x=1$, but the $(k+1)$st derivative does not.  Let $1 \le \ell \le k$.  Consider the $\ell$th derivative of $f$ evaluated at $x=1$:
\begin{align*} f^{(\ell)} = \sum_{i=1}^s (a_i - a) ( a_i - 1 - a) \ldots (a_i - (\ell - 1) - a) - (b_i - a) ( b_i - 1 - a) \ldots (b_i - (\ell - 1) - a)
 \\ =  \sum_{i=1}^{s} a_i^{\ell} - b_i^{\ell} + g_{k-1} (a_i^{\ell-1} - b_i^{\ell-1}) + \ldots + g_1 (a_i - b_i) + g_0\end{align*}
where $g_i$ is some function depending on $i$ and $a$.  Since the $a_i, b_i$ satisfy the conditions of Lemma \ref{wooleylemma}, the $\ell$th derivative is equal to zero if $1 \le \ell \le k$.  Moreover, the $(k+1)$st derivative of $f$ at $x=1$ then simplifies to
\[ f^{(k+1)} = \sum_{i=1}^s a_i^{k+1} - b_i^{k+1} \neq 0. \]
So $f$ has a zero at $x=1$ of order precisely $k$.
\end{proof}
%
%
\begin{lemma}\label{dyadicsums}
For every $k \in \mathbb{N}, \epsilon > 0$, there exists an $n_0 = n_0 (k, \epsilon)$ such that if $A \subseteq \mathbb{R}$ is of size $n \ge n_0$, and no dyadic interval $[x,2x)$ contains more than $s$ elements of $A$.  Then,
\[ |kA| \gg_k   \frac{|A|^k}{s^k}  \]
\end{lemma}
\begin{proof}Without loss of generality, we may assume half the elements of $A$ are nonnegative, else, replace $A$ with $-A$ and repeat the proof since $|kA| = |k(-A)|$.  Denote the nonnegative elements as $A' := \{ a_1 < \ldots < a_n \}$, and let 
\[ B := \{ a_{2s}, a_{4s}, a_{6s} , \ldots , a_{(2 \lfloor \frac{n}{2s} \rfloor )s} \}. \]
Now, consider $kB$.  Suppose
\begin{equation}\label{b1} b_1 + \ldots + b_k = b_1' + \ldots + b_k'. \end{equation}
 for some $b_1 < \ldots < b_k, b_1' < \ldots < b_k' \in B$.  We claim that this implies $b_i = b_i'$ for all $i = 1, \ldots , k$.  Let $t \in \{1, \ldots , k\}$ be the largest integer such that $b_t \neq b_t'$.  Without loss of generality, if $b_t > b_t'$, then in fact $b_t > 2 b_t'$ since they belong to nonconsecutive dyadic intervals.  Moreover,
\[ b_1' + \ldots + b_{t-1}' + b_t' \le b_t' + b_t' < b_t < b_1 + \ldots + b_t. \]
Hence, all the sums $b_1 + \ldots + b_k$ are unique, and so
\[ |kA| \ge |kB| = \binom{|B|}{k} \gg_k |B|^k \gg_k  \frac{|A|^k}{s^k}. \] 
\end{proof}
%
%
Let $C \subseteq \mathbb{R}$.  We call $C_0, \ldots , C_{k-1}$ a \textit{decreasing partition} of $C$ if
\[ C = \bigcup_{i=0}^{k-1} C_i \]
and for any distinct $i,j \in \{ 0, \ldots , k-1\}$, if $i < j$, then $|c| > |d|$ for all $c \in C_i, d\in C_j$.
\begin{lemma}\label{deltasums}Suppose that $C \subseteq \mathbb{R}- \{0\}$, and let
\[ 1 = \delta_0 > \delta_1 > \ldots > \delta_{k-1} > 0. \]
Moreover, suppose that $C$ has the property that for any $c > d \in C$,
\begin{equation}\label{assump1} \frac{c}{d} -1 > 2k \frac{\delta_i}{\delta_{i-1}}. \end{equation}
for all $i = 1, \ldots , k-1$.  Then for any decreasing partition $C_0, \ldots , C_{k-1}$ of $C$, then we must have that the sums
\[  c_0 \delta_0 + c_1 \delta_1 + \ldots + c_{k-1} \delta_{k-1} \]
are distinct for all $(c_0, c_1, \ldots , c_{k-1}) \in C_0 \times C_1 \times \ldots \times C_{k-1}$.
\end{lemma}
\begin{proof}
Suppose 
\begin{equation}\label{ceq}\sum_{i=0}^{k-1} c_i \delta_i = \sum_{i=0}^{k-1} c_i' \delta_i \end{equation}
where $c_i, c_i' \in C_i$. Let $j$ be the smallest integer in $\{ 0, \ldots , k-1\}$ such that $c_j \neq c_j'$.  Hence, we need only consider
\begin{equation}\label{ceq2}\sum_{i=j}^{k-1} c_i \delta_i = \sum_{i=j}^{k-1} c_i' \delta_i \end{equation}
We will now derive a contradiction proving no such $j$ exists and so \eqref{ceq} only holds when $c_{i} = c_i'$ for all $i$.  For a contradiction, suppose $c_j > c_j'$.  Dividing by $c_j' \delta_j$ on both sides and rearranging, the sum becomes
\[ \frac{c_j}{c_j'} - 1 = \sum_{i=j+1}^{k-1} \frac{c_i' - c_i}{c_j'} \cdot \frac{\delta_i}{\delta_j}. \]
By \eqref{assump1}, this implies that 
\[ \sum_{i=j+1}^{k-1} \frac{c_i' - c_i}{c_j'} \cdot \frac{\delta_i}{\delta_j} > 2k \frac{\delta_{j+1}}{\delta_{j}}. \]
On the other hand, since the $C_i$ form a decreasing partition, 
\[ \left | \frac{c_i' - c_i}{c_j'} \right | < 2 \]
for all $i \ge j+1$. Also, since $\delta_{j+1} > \delta_{\ell} > 0$ for all $\ell > j+1$
\[ \frac{\delta_i}{\delta_j} < \frac{\delta_{j+1}}{\delta_j}. \]
So we get a contradiction since this would imply
\[ \left | \sum_{i=j+1}^{k-1} \frac{c_i' - c_i}{c_j'} \cdot \frac{\delta_i}{\delta_j} \right | < 2k \frac{\delta_{j+1}}{\delta_j}. \]
\end{proof}
\section{Finding a Long Geometric Progression in $A/A$}
The following two lemmas are variants of Lemma 2 in the work of Croot and Hart \cite{croothart}.  The first one is a repackaged version of the main idea in \cite{croruzscho} which allows one to combinatorially find long progressions in difference (or quotient) sets.  The second lemma builds upon the first by taking $(N+1)$-tuples and showing that one can project them in a way that satsfies properties we will need later on.   
%
%
\begin{lemma}\label{findingap}
For all $N \in \mathbb{N}$, $\epsilon > 0$, if $B \subseteq A \subseteq \mathbb{R}$ such that $|A.A|< |A|^{1+\epsilon}$, then the following holds.  There exists $\alpha \in \mathbb{R}$ and $\theta \in \frac{B}{B}$ such that there are $|A|^{N+2 - 7\epsilon N^2}$ tuples $(a, y_0, \ldots , y_N) \in A^{N+2}$ such that
\[ a y_i \theta^i \in \alpha A \]
for all $i = 0, \ldots , N$.
\end{lemma}
\begin{proof}
Let $\epsilon > 0$, and let $B \subseteq A \subseteq \mathbb{R}$ be such that $|A.A| < |A|^{1+\epsilon}$.  Consider the following set:
\begin{align*}E := \{( b_1, b_2, a_1, a_2, u, v, y_0, \ldots , y_N, z_0, \ldots , z_N) \in B^2 \times A^{2N+6}:  v a_1 b_1^i z_i = u a_2 b_2^i y_i \text{ $i = 0, \ldots , N$}\}.
\end{align*}
For a vector $\textbf{t} = (t_0, \ldots , t_N) \in A^{(3)} \times A^{(4)} \ldots \times A^{(N+3)}$, let
\[ r \left (\textbf{t} \right) := |\{ (b, v, a, z_0, \ldots , z_N) \in B \times A^{N+3}: v a b^i z_i = t_i \text{ for $i = 0, \ldots, N$} \}|. \]
Note that here is where we use the fact that $B \subseteq A$ in order to assume that $v a b^i z_i \in A^{(i + 3)}$.  Now, one can use the Cauchy-Schwarz inequality to bound the size of $E$:
\[ |E| = \sum_{\textbf{t}} r(\textbf{t})^2 \ge \frac{|B|^2|A|^{2N+6}}{|A^{(3)}||A^{(4)}| \ldots |A^{(N+3)}|}  \]
By the Pl\"unnecke-Ruzsa inequality, since $|A| < |A|^{1+\epsilon}$, we then have that for all $i$, $|A^{(i)}| < |A|^{1+i\epsilon }$.  Thus,
\[ |E| \ge |B|^2 |A|^{N+5 - \epsilon ( 3 + 4 + \ldots + N+3)} \ge |B|^{2} |A|^{N+5-7\epsilon N^2}. \]
By the pigeonhole principle, there exists a $(b_1, b_2, u, v, a_1) \in B^2 \times A^3$ such that there are $|A|^{N+2 - 7\epsilon N^2}$ tuples $(a_2, y_0 , \ldots , y_N , z_0 , \ldots , z_N)$ such that for $i= 0 , \ldots , N$
\[ v a_1 b_1^i z_i = u a_2 b_2^i y_i. \]
Rearranging the above, we get that
\[ z_i = a_2 \frac{u}{v a_1}  \left ( \frac{b_2}{b_1}\right )^i y_i. \]
Letting $\alpha = \frac{v a_1}{u}$ and $\theta = \frac{b_2}{b_1}$ proves the lemma.
\end{proof}
\begin{lemma}\label{findingap2}
Let $N, \ell \in \mathbb{N}$, $\epsilon > 0$ and let $c = 2 \ell^{\lceil \log_2{N} \rceil}$. There exists an $n_0=n_0(N, \ell, \epsilon)$ such that if $A \subseteq \mathbb{R}$ is of size $n \ge n_0$ then the following holds.  If $|A.A| \le |A|^{1+\epsilon}$, then for any $B \subseteq A$ there exists $Y_0, \ldots , Y_N \subseteq A$ such that 
\begin{enumerate}
\item $|Y_i| \ge |A|^{1-O(\epsilon c N^4)}$.
\item For any collection of subsets $Y_i' \subseteq Y_i$ satisfying $|Y_i'| \le c$ then there exists an $\alpha \in \mathbb{R}$, $\theta \in \frac{B}{B}$, and an $A' \subseteq A$ of size at least $\sqrt{|A|}$ such that
\[ a y_i \theta^i \in \alpha A \]
for all $a \in A'$, $y_i \in Y_i'$, $i \in \{0, \ldots , N\}$.
\end{enumerate}   
\end{lemma}
We first need a graph theoretic lemma.  It is a slight variant of a lemma found in the excellent survey by Fox and Sudakov about the technique of dependent random choice \cite{foxsudakov}.  For a graph $G$ and $T \subseteq G$, let $\Gamma(T)$ denote the set of common neighbors of $T$; that is, the set of all vertices adjacent to every vertex in $T$.
%
%
\begin{lemma}\label{deprc}
Let $a,m,r \in \mathbb{N}$.  Let $G=[X,Y]$ be a bipartite graph with $|E(G)|$ edges.  If there exists a $t \in \mathbb{N}$ such that
\[ \frac{|E(G)|^t}{|X|^t |Y|^{t-1}} - \binom{|Y|}{r} \left ( \frac{m}{|X|} \right )^t \ge a \]
then there exists a set of vertices in $Y$ of size $a$ such that every $r$ of them have at least $m$ common neighbors.
\end{lemma}
\begin{proof}
Let $T \subseteq X$ be a set of $t$ vertices chosen uniformly at random with repetition.  Let $\Gamma(T)$ denote the set of common neighbors of $T$, and let $Z = |\Gamma(T)|$.  Then, by linearity of expectation and H\"older's inequality
\[ \mathbb{E}(Z) = \sum_{y\in Y} \mathbb{P}(T \subseteq N(y)) = \sum_{y \in Y} \left (\frac{|N(y)|}{|X|} \right )^t \ge \frac{|E(G)|^t}{|X|^t |Y|^{t-1}}. \]
Now, let $W$ be the random variable associated to the number of sets of $r$ vertices in $\Gamma(T)$ with less than $m$ common neighbors.  We want $W$ to be small so that we may modify all these deficient sets and prove the lemma.  First, note that for any set $S \subseteq Y$ of size $r$ with less than $m$ common neighbors, the probability that $S$ is also a subset of $\Gamma(T)$ is
\[ \left ( \frac{|\Gamma(S)|}{|X|} \right )^t \]
since the only way that $S \subseteq \Gamma(T)$ is if every vertex from the common neighborhood of $S$ gets chosen in $T$.  Hence,
\[ \mathbb{E}(W) \le  \left ( \frac{|\Gamma(S)|}{|X|} \right )^t \binom{|Y|}{r} < \frac{m^t}{|X|^t} \binom{|Y|}{r}. \]
Therefore, there exists a choice of $T$ such that
\[ \mathbb{E}(Z-W) > \frac{|E(G)|^t}{|X|^t |Y|^{t-1}} - \binom{|Y|}{r} \left ( \frac{m}{|X|} \right )^t \ge a. \]
Let $T$ be chosen such that the above holds.  For each set $S \subseteq \Gamma(T)$ of size $r$ with less than $m$ common neighbors, remove a vertex arbitrarily from $S$.  After this process, $\Gamma(T)$ still has at least $a$ vertices left, and every set of size $r$ has at least $m$ common neighbors. 
\end{proof}
\begin{proof}[Proof of \cref{findingap2}]
Apply Lemma \ref{findingap} to get an $\alpha \in \mathbb{R}$ and a $b_1 > b_2 \in B$ such that there are $|A|^{N+2 - 7\epsilon N^2}$ tuples 
\[ T := (a, y_0, \ldots , y_N) \in A^{N+2} \]
such that
\begin{equation}\label{a1eq} \alpha a y_i \theta^i \in A  \text{ for $i = 0, \ldots , N$.} \end{equation}
Let $G[X,Y]$ be the bipartite graph defined by $X = A$, $Y = A^{N+1}$, and edges defined by the set $T$.  Observe that for any constant $r$ depending only on $\ell$ and $N$ there exists a $t$ and an $\epsilon$ such that if $A$ is sufficiently large, then
\[  \frac{|A|^{t(N+2- 7\epsilon N^2)}}{|A|^t |A|^{(t-1)(N+1)}} - \binom{|A|^{N+1}}{r}\left ( \frac{|A|^{t/2}}{|A|^{t}} \right ) \ge 
    |A|^{N+1 - 7 \epsilon t N^2} - |A|^{r(N+1) - t/2} \ge \frac{1}{2}|A|^{N+1 - 7 \epsilon t N^2}. \]
In particular, one may choose $t = 2r(N+1)$.  Hence, we may apply Lemma \ref{deprc} with $a = \frac{1}{2}|A|^{N+1 - 14 \epsilon r N^3}$, $m = |A|^{1/2}$, and $r = c(N+1)$.  Let $Y'\subseteq Y$ denote the set found by Lemma \ref{deprc} with the specified property.

Each vertex $v \in Y'$ is associated to a corresponding $(N+1)$-tuple; for $i = 0 , \ldots , N$, let $Y_i$ be the projection of $Y'$ onto the $i$th coordinate axis.  One can see that $|Y_i| \ge |A|^{1 - O(\epsilon c N^4)}$.  Consider an arbitrary collection of subsets $Y_i' \subseteq Y_i$ satisfying $|Y_i'| \le c$.  Let $y_{i,j} \in Y_i'$.  Our goal is to show there is a fixed set $A' \subseteq A$ of $|A|^{1/2}$ elements such that \eqref{a1eq} holds for all $y_{i,j}$, $a\in A'$, $i = \{0, \ldots , N\}$.  

  Since $y_{i,j} \in Y_i' \subseteq Y_i$, there exists a corresponding $(N+1)$-tuple 
\[ (u_0, u_1, \ldots , u_{i-1}, y_{i,j}, u_{i+1}, \ldots , u_N) \in Y' . \]
  For each $y_{i,j}$, arbitrarily choose such a tuple in $Y'$, and denote the tuple as $v_{i,j}$.  Let $V$ be the collection of all such $v_{i,j}$.  So, letting $|V| \le c(N+1)$ be the constant $r$ in the application of \cref{deprc}, we can conclude that there is a set of $|A|^{1/2}$ vertices in $X$ adjacent to every vertex in $V$.  Let $A'$ be this set of $|A|^{1/2}$ vertices.  Hence, there is a set of $|A|^{1/2}$ elements such that for any $y_{i,j} \in Y_i'$ \eqref{a1eq} holds for all $a \in A'$, $i \in \{0, \ldots , N\}$.
\end{proof}
\section{Intersections of Multifold Sumsets}
%
We now prove the following lemma that gives us information when lots of multifold sumsets intersect trivially.  This lemma is what introduces a significant amount of loss in the strength of our overall bound in \cref{mainthm} -- that is, it is the main obstruction in improving the exponent $\exp(c \sqrt{\log{h}})$ to some fixed power of $h$.
\begin{lemma}\label{intersections}
Let $A\subseteq \mathbb{R}$ be of size $n$ and $\ell, t \in \mathbb{N}$.  Let $A_i \subseteq A$ for $i = 1, \ldots , 2^t$ be such that
\[ \bigcap_{i=1}^{2^t} f(t,i) \ell^{g(t,i)} A_i - f(t,i) \ell^{g(t,i)} A_i = \{0 \}. \]
Then, there exists an $i \in \{ 2, \ldots , t+1\}$ and an $j \in \{1, \ldots , 2^t \}$ such that
\[ |(\ell^{i-1} + \ell^{i})A| \ge n^{\frac{1}{3^{t+1}}} |\ell^i A_i|. \]
\end{lemma}
%
The functions $f$ and $g$ in the above lemma are defined as follows.  For $a \in \mathbb{N}$, $b = 1,\ldots 2^{a}$, define $f(a,b)$ recursively as follows: 
\begin{align}
\nonumber & f(1,1) := 1, f(1,2) := 2 ,  \\
\label{propertyf1}& f(a,2b-1) := f(a-1,b); b= 1, \ldots , 2^{a-1} \\
\label{propertyf2}& f(a,2b) := 2 f(a,2b-1) = 2f(a-1,b) ; b = 1, \ldots , 2^{a-1}
\end{align}
For the benefit of the reader, we list the first few values of $f(a,b)$:
\begin{align*}
& f(1,1) = 1 ; f(1,2) = 2 \\
& f(2,1) = 1; f(2,2) = 2; f(2,3) = 2; f(2,4) = 4 \\
& f(3,1) = 1; f(3,2) = 2; f(3,3) = 2; f(3,4) = 4; f(3,5) = 2; f(3,6) = 4; f(3,7) = 4; f(3,8) = 8
\end{align*}
Observe that
\begin{equation}\label{propertyf3}
f(a,b) = 2^k \text{ for some } k \le a. 
\end{equation}
Denote $g(a,b) := \log_2{f(a,b)} + 1$.  Observe that by \eqref{propertyf2}, 
\begin{equation}\label{propertyg1} g(a, 2b) = g(a,2b-1) + 1 \end{equation}
and by \eqref{propertyf3}, 
\begin{equation}\label{propertyg2}
g(a,b) \le a+1.
\end{equation}
%
%
%
The following covering lemma, which is potentially of independent interest, is the main tool in proving \cref{intersections}.
\begin{lemma}\label{2caselemma}
For any $X, Y$ in an abelian group $G$ and any $K \in \mathbb{N}$, there exists an $X' \subseteq X$ such that either 
\begin{enumerate}
\item $|X'| \ge K$ and $X'-X' \cap Y - Y = \{ 0 \}$, or 
\item $|X'| \ge  \frac{|X|}{K}$ and $X'-X' \subseteq 2Y -2 Y$.
\end{enumerate}
\end{lemma}
This follows quickly from the following graph theory lemma.
\begin{lemma}
For any graph $G$ and any $0 \le K \le |G|$, $G$ contains an indepdendent set of size at least $K$ or a vertex of degree at least $|G|/K$. 
\end{lemma}
\begin{proof}
If $G$ has a vertex of degree at least $|G|/K$, we are done.  Hence, the maximum degree of $G$, $\Delta(G)$ is less than $|G|/K$.  By the greedy algorithm, we can find an independent set of size
\[ K \ge \left \lfloor \frac{|G| + \Delta}{\Delta + 1} \right \rfloor. \]
\end{proof}
\begin{proof}[Proof of \cref{2caselemma}]
Let $G = (V,E)$ be the graph defined by $V(G) := X$ and $\{ u, v \} \in E(G)$ if $u-v \in Y-Y$.  Observe that since $Y-Y$ is symmetric, these edges are undirected.  If $G$ contains an independent set $X'$ of size at least $K$, for any distinct $u,v \in X'$, $u-v \notin Y-Y$.  Hence, $X'-X' \cap Y-Y = \{0\}$.  Otherwise, $G$ contains a vertex, $a$, of degree at least $\frac{|X|}{K}$.  Letting the neighborhood of this vertex be $X'$, for any $u, v \in X'$, $\{u, a\}$ and $\{a,v \}$ are edges.  Since $u - v = u-a + a-v$, we have that $u - v \in 2Y - 2Y$.
\end{proof}
%
%
\begin{proof}[Proof of \cref{intersections}]
We perform the following algorithm to find such an $i, j$ as in the conclusion of the lemma.  We outline steps $j=0, \ldots , t-2$.    

\textbf{Step 0: } Let $A_{0,i} := \ell^{g(t,i)} A_i$.  For $i = 1, \ldots , 2^{t-1}$, apply \cref{2caselemma} with 
\[ X := A_{0,2i-1}, \text{ } Y := A_{0,2i} , \text{ and } K := K_0 = n^{\frac{1}{3^t}}, \] and observe which case holds.  If for any $i$, Case 1 holds, we halt since this implies that there exists an $X' \subseteq X$ with $|X'| \ge n^{\frac{1}{3^t}}$ and
\[ |(\ell^{g(t,2i-1)} + \ell^{g(t,2i)})A | \ge | A_{0, 2i-1} + A_{0,2i} | \ge |X' + Y| = |X'||Y| \ge n^{\frac{1}{3^t}} |\ell^{g(t,2i)} A_{2i}|. \]
This satisfies the conclusion of the lemma with $k=g(t,2i)$ and $j=2i$.  Hence, we may assume Case 2 holds for all $i$.  Therefore, there exists an $X' \subseteq X$ such that $X' - X' \subseteq 2Y - 2Y$.  Adding $X'-X'$ to itself multiple times also implies for any positive integer $s$, $sX' - sX' \subseteq 2sY - 2sY$.  In particular for $s = f(t,2i-1)$,
\begin{align} 
f(t,2i-1)  X' - f(t,2i-1)  X' & \subseteq 2f(t,2i-1)  Y - 2 f(t,2i-1) Y  \nonumber
							    \\ \label{aligneq1a} & = 2f(t,2i-1) A_{0,2i} - 2f(t,2i-1) A_{0,2i}  
							    \\ & = f(t,2i) A_{0,2i} - f(t,2i) A_{0,2i}.  \nonumber
\end{align}
where we used \eqref{propertyf2} in the last equality.  Also,
\begin{align}\label{aligneq2a}\begin{split}
f(t,2i-1) X' - f(t,2i-1) X' & \subseteq f(t,2i-1) X - f(t,2i-1) X
							\\  & = f(t,2i-1) A_{0,2i-1} - f(t,2i-1) A_{0,2i-1}
\end{split}\end{align}
Letting $A_{1, i} := X'$, we then have that by \eqref{propertyf1}, \eqref{aligneq1a}, and \eqref{aligneq2a}
\begin{align*} \bigcap_{i=1}^{2^{t-1}} f(t-1,i) A_{1,i} - f(t-1,i) A_i & \subseteq \bigcap_{i=1}^{2^t} f(t,i) A_{0,i} - f(t,i) A_{0,i} \\
										& =  \bigcap_{i=1}^{2^t} f(t,i) \ell^{g(t,i)} A_i - f(t,i) \ell^{g(t,i)} A_i = \{0 \}.  \end{align*}
And we also have that 
\[ |A_{1,i}| \ge \frac{|A_{0,2i-1}|}{K_0} \]
The next steps, Steps $j=1, \ldots , t-2$, are iterations of this argument with a very slight change in the choice of $X$ and $Y$ in the application of \cref{2caselemma}.  


\textbf{Step j: } Let $A_{j,i} \subseteq A_{j-1, 2i-1}$ be as specified in Step (j-1) of the algorithm.  In particular, $A_{j,i}$ satisfies 
\[ |A_{j,i}| \ge \frac{|A_{j-1,2i-1}|}{K_{j-1}}. \] 
An easy inductive argument shows that there exists an $s$ such that
\begin{equation}\label{upthetree}
A_{j,i} \subseteq A_{j-1, 2i-1} \subseteq \ldots \subseteq A_{0,s} \subseteq \ell^{g(t-j,i)} A_s.
\end{equation} 
where we draw the reader's attention to the fact that the subscript $A_{j,i}$ determines the exponent at the end, $g(t-j,i)$.  For $i = 1, \ldots , 2^{t-j-1}$, apply \cref{2caselemma} with $X = A_{j,2i-1}$, $Y = A_{j,2i}$, $K := K_j = n^{\frac{1}{3^{t-j}}}$, and observe which case holds.  If for any $i$, Case 1 holds, we halt since by \cref{easyint} this implies that
\begin{align}\label{RHSChain1}
\begin{split}
|X' + Y| = |X'||Y| & \ge K_j | A_{j,2i}| 
			\\ & \ge \frac{K_j}{K_{j-1}} |A_{j-1, 4i-1}|
			\\ & \vdots
			\\ & \ge \frac{K_j}{K_{j-1} K_{j-2} \ldots K_0} |A_{0, s} |
			\\ & = n^{\frac{1}{3^{t+1}}} |\ell^{g(t-j,2i)} A_s |
\end{split}
\end{align}
for some integer $s$.  On the other hand, using \eqref{upthetree} and \eqref{propertyg1}, we have
\begin{align}\label{LHSChain1}
\begin{split}
|X' + Y| \le |X+Y| & = | A_{j,2i-1} + A_{j,2i}|
			\\ & \le |A_{j-1, 4i-3} + A_{j-1,4i-1} |
			\\ & \vdots
			\\ & \le |A_{0, s_1} + A_{0,s} |   \le |(\ell^{j} + \ell^{j+1})A|
\end{split}
\end{align}
for $j=g(t-j,2i-1)$.  Combining \eqref{RHSChain1} and \eqref{LHSChain1} shows that we have satisfied the conclusion of the Lemma.  


Hence, we may assume Case 2 holds for all $i$.  Therefore, there exists an $X' \subseteq X$ with $|X'| \ge |X|/K_{j}$ such that $X' - X' \subseteq 2Y - 2Y$.  Moreover, for any positive integer $s$, $sX' - sX' \subseteq 2sY - 2sY$.  For $s = f(t-j-1,i)$
\begin{align} 
f(t-j-1,i)  X' - f(t-j-1,i) X' & \subseteq 2f(t-j-1,i)  Y - 2 f(t-j-1,i) Y  \nonumber
							    \\ \label{aligneq1} & = 2f(t-j-1,i) A_{j,2i} - 2 f(t-j-1,i) A_{j,2i}  
							    \\ & = f(t-j,2i) A_{j,2i} - f(t-j,2i) A_{j,2i}.  \nonumber
\end{align}
where we used \eqref{propertyf1} in the last equality.  Also,
\begin{align}\label{aligneq2}\begin{split}
f(t-j-1,i) X' - f(t-j-1,i) X' & \subseteq f(t-j-1,i)  X - f(t-j-1,i)  X
							\\  & = f(t-j,2i-1)  A_{j,2i-1} - f(t-j,2i-1) A_{j,2i-1} 
\end{split}\end{align}
Letting $A_{j+1, i} := X'$, we then have that by \eqref{propertyf2}, \eqref{aligneq1}, and \eqref{aligneq2}
\[ \bigcap_{i=1}^{2^{t-j-1}} f(t-j-1,i)  A_{j+1,i} - f(t-j-1,i) A_{j+1,i} \subseteq \bigcap_{i=1}^{2^{t-j}} f(t-j,i) A_i - f(t-j,i) A_i = \{0 \}. \]
We now proceed to Step j+1 with $A_{j+1,i}$, $i = 1, \ldots , t-j-1$.

\textbf{Step t -- 1: } If we have not halted, then at this point, we only have 2 sets, $A_{t-1,1}, A_{t-1,2}$, such that
\begin{align*}
f(1,1)  A_{t-1,1} - f(1,1)  A_{t-1,1} \cap f(1,2) A_{t-1,2} - f(1,2)  A_{t-1,2} =  \{ 0 \} . 
\end{align*}
Since $f(1,1)=1$, $f(1,2)=2$, and 
\[  A_{t-1,1} - A_{t-1,1} \cap A_{t-1,2} - A_{t-1,2} \subseteq  A_{t-1,1} - A_{t-1,1} \cap 2 A_{t-1,2} - 2  A_{t-1,2} = \{0 \} \]
we then have by \cref{easyint}
\[ |A_{t-1,1} + A_{t-1,2}| = |A_{t-1,1}| |A_{t-1,2}|. \]
Tracing back our steps in the algorithms as we did in \eqref{RHSChain1} and \eqref{LHSChain1}, we get that
\begin{align}\label{RHSChain2}
\begin{split}
|A_{t-1,1}| |A_{t-1,2}| \ge \frac{|A_{t-2,1}||A_{t-2,3}|}{K_{t-1}^2} 
			 & \ge \frac{|A_{t-3,1}||A_{t-3,5}|}{K_{t-1}^2 K_{t-2}^2}
			\\ & \ge \frac{|A_{t-3,1}||A_{t-3,9}|}{ K_{t-1}^2 K_{t-2}^2 K_{t-3}^2}
			\\ & \vdots
			\\ & \ge \frac{|A_{0,1}||A_{0,2^{t-1}+1}|}{K_{t-1}^2 K_{t-2}^2 K_{t-3}^2 \ldots K_0^2}
			\ge  n^{\frac{1}{3^{t+1}}} |\ell^2 A_{2^{t-1}+1}|
\end{split}
\end{align}
Note that we used the fact that $|A_{1,0}| \ge n$ in the last inequality.
On the other hand,
\begin{align}\label{LHSChain2}
\begin{split}
|A_{t-1,1} + A_{t-1,2}| \le |A_{t-2,1}+A_{t-2,3}|
			 & \le |A_{t-3,1}+A_{t-3,5}|
			\\ & \vdots
			\\ & \le |A_{0,1}+A_{0, 2^{t-1}+1}|  
			\\ & \le |\ell A_1 + \ell^2 A_{2^{t-1}+1}| \le |(\ell + \ell^{2})A|
\end{split}
\end{align}
Combining \eqref{RHSChain2} and \eqref{LHSChain2} completes the proof of the lemma.
\end{proof}
\section{Proof of Main Theorem}
The proof of our main theorem is iterative.  The argument splits into two cases: in one case, we prove our bound directly similar to \cite{croothart}; the other case we have to iteratively use \cref{intersections} to get a small amount of growth each iteration while passing to subsets of our original set.  After enough iterations, we prove our bound.
\begin{proposition}\label{mainprop}
Let $h \in \mathbb{N}$.  Let 
\[ e^{ \sqrt{ \frac{1}{100} \log{ \frac{h}{2}}  } } \text{ and } \ell := k^{8}. \]
There exists an $\epsilon':= \epsilon'(h)$ such that for any $0 < \epsilon < \epsilon'$ there exists an $n_0 := n_0(\epsilon,h)$ such that if $A \subseteq \mathbb{R}$ is of size $n \ge n_0$ and $|A.A| \le |A|^{1+\epsilon}$, then either
\[ |h A| \ge |A|^{\Omega(k)} \]
or there exists an $A' \subseteq A$ and a $c:= c(h)$ such that $|A'| \ge |A|^{1-c\epsilon}$, and
\[ |(\ell^j + \ell^{j-1})A| \gg_{h} |A|^{\frac{1}{22k^{6}}} |\ell^j A'| \]
for some $j \in \{2, \ldots , \log{8k^5} \}$.
\end{proposition}
\begin{proof}[Proof of \cref{mainprop}]
Let $A \subseteq \mathbb{R}$ be such that $|A.A| \le |A|^{1+\epsilon}$.  Let $k, \ell$ be constants depending on $h$ as specified in the statement of the proposition.  Apply Corollary \ref{wooleycor} to get a set of polynomials $f_j(x)$ for $j= 2 ,\ldots , k-1$ such that each polynomial has coefficients in $\{-1,0,1\}$, $f_j(x)$ has a root at $x=1$ of order exactly $j$, and $f_j(x)$ has at most $j^2 \le k^2$ nonzero terms.  Let $f_0(x) := 1$, and $f_1(x) := x-1$.
\[ N := \max_j \{ deg(f_j ): j = 0 , \ldots , k-1 \} \]
and let $S\subseteq \{0, \ldots , N\}$ be such that $i \in S$ if and only if there is an $f_j(x)$ such that the coefficient of $x^i$ is nonzero.  Let $M := |S|$ and observe that $M \le k^3$. 

Denote $A := \{ a_1 < \ldots < a_n\}$, let $0 < \delta < 1/4$ be a parameter chosen later, and let $s := \lfloor n^{\delta} \rfloor$.  Let 
\[ B' := \{ a_{i}, a_{i+1}, \ldots , a_{i+s-1} \} \]
be chosen such that $a_{i+s-1}/a_i$ is minimal.  By Lemma \ref{dyadicsums}, if no dyadic interval contains more than $s$ elements of $A$, we are done.  Hence, $B' \subseteq [x, 2x)$ for some $x \in \mathbb{R}$.  Let $0 < \gamma < 1$ be a small constant depending on $h$ to be chosen later.  There exists a subinterval
\[ [y, y+ \gamma x) \subseteq [x,2x) \]
with at least $\gamma s$ elements of $A$ in it.  Let $B$ be the intersection of $A$ with this subinterval.  So $B \subseteq A$ has the properties that $|B| \ge \gamma s$ and for any $b, b' \in B$, 
\[ \left | \frac{b}{b'} - 1 \right | < \gamma . \]
The latter property will be important when we later consider polynomials with roots at $1$ evaluated at $\frac{b}{b'}$.  

Apply \cref{findingap2} with $N, \ell, \epsilon, B$ to find a set of $Y_i \subseteq A$, $\alpha \in \mathbb{R}$, $\theta \in B/B$, satisfying the conclusion of the lemma.  We will discard some of the sets from $Y_0,\ldots , Y_N$ in the following way.  If $i \notin S$, then we throw out $Y_i$.  Abusing our notation, relabel the remaining sets as $Y_1,\ldots , Y_M$.  Let $t = \lceil \log_2{M} \rceil \le \lceil \log_2{k^3} \rceil$.  If 
\[ \bigcap_{i=1}^M \ell^{t}Y_i - \ell^{t} Y_i = \{0\} \]
then we may apply \cref{intersections} to conclude that there exists an $i \in \{2, \ldots , t+1\}$ and a $j \in \{1, \ldots , 2^t\}$ such that
\[ |(\ell^{i-1}+\ell^{i})A| \ge |A|^{\frac{1}{3^{t+1}}} |\ell^{i}Y_i| \ge |A|^{\frac{1}{22k^{6}}} |\ell^i Y_i|. \]
This satisfies the second conclusion of the proposition, so we may assume that there exists a nonzero $\beta$ in the above intersection.  That is, a nonzero $\beta$ such that for $i = 1, \ldots , M$,
\[ \beta = \sum_{j=1}^{\ell^t} y_{i,j} - \sum_{j = \ell^t+1}^{2\ell^t} y_{i, j} \]
where $y_{i,j} \in Y_i$.  Letting $Y_i' := \{ y_{i,j}:  j = 1, \ldots , 2\ell^t \}$, by the conclusion of \cref{findingap2}, there exists an $A' \subseteq A$ of size at least $|A|^{1/2}$ such that 
\begin{equation}\label{amazing} a y_{i,j} \theta^i \in \alpha A \text{ for $i = 1, \ldots , M$, and any $a\in A'$.}\end{equation}

Denote $A' := \{ a_1 < a_2 < \ldots < a_{|A'|} \}$, and let $C := \{ a_{i_1}, a_{i_2}, \ldots , a_{i_r} \}$ where 
\[ i_j = j \lfloor n^{1/4} \rfloor \text{ and } r = \left \lfloor \frac{|A'|}{n^{1/4}} \right \rfloor. \]
This ensures that we have
\begin{equation}\label{wellspacing} \frac{c}{c'} > \theta \text{ for any $c, c' \in C$} \end{equation}
by our choice of $B'$ along with the fact that $s < \lfloor n^{1/4} \rfloor $.  Decompose $C$ into $C_0, C_1, \ldots , C_{k-1}$ where all elements of $C_i$ are greater than all elements of $C_j$ for $i < j$, and for all $i = 0, \ldots , k-2$, $|C_i| = \lfloor |C|/k \rfloor$.  For $i = 0, \ldots , k-1$, let $\delta_i := f_i ( \theta)$.  Now consider sums of the form
\begin{equation}\Sigma = \label{eqdeltasums} \{ \beta ( c_0 \delta_0 + c_1 \delta_1 + \ldots + c_{k-1} \delta_{k-1} ): c_i \in C_i  \} . \end{equation}
We verify that $C$ and $\delta_i$ satisfy the requirements of Lemma \ref{deltasums} as follows.  Since
\[ \frac{\delta_i}{\delta_{i-1}} = \frac{f_i(x)}{f_{i-1}(x)} = (x-1) g_i(x) \]
where the coefficients of $g_i$ depend only on $k$, we may choose $\gamma$ small enough such that 
\[ \theta - 1 < \frac{1}{g_i(\theta)}. \]
So we have that $\delta_{i-1} > \delta_{i}$ for all $i = 1, \ldots , k-1$.  Let $c, d\in C$.  From \eqref{wellspacing}, we have that $\frac{c}{d} > \theta$.  However, by choosing $\delta$ small enough, we can assume that in fact $\frac{c}{d} > \theta^{r}$ for any $r = r(k)$.  Hence,
\[ \frac{c}{d} - 1 >  \theta^r - 1 = (\theta - 1) ( 1 + \theta + \ldots + \theta^{r-1}) \ge (\theta - 1)r. \]
By choosing $r > 2k \cdot g_i(\theta)$, we have
\[ \frac{c}{d}-1 \ge (\theta - 1)2k \cdot g_{i}(\theta) = 2k \frac{\delta_i}{\delta_{i-1}}. \]  
So by Lemma \ref{deltasums}, all the sums of the form \eqref{eqdeltasums} are distinct, and so 
\[ |\Sigma| \ge \prod_{i=0}^{k-1} |C_i| .\]
  We can rewrite \eqref{eqdeltasums} by grouping like powers of $\theta$ as
\[ \beta \left [ \left ( \sum_{i=0}^{k-1} \epsilon_{0, i} c_i \right ) \theta^0  +  \left ( \sum_{i=0}^{k-1} \epsilon_{1, i} c_i \right ) \theta^1 + \ldots + \left ( \sum_{i=0}^{k-1} \epsilon_{M, i} c_i \right ) \theta^{N} \right ] \]
where $\epsilon_{i,j} \in \{-1,0, 1\}$.
Recall that $S$ is the set of powers of $\theta$ that have at least one nonzero coefficient in some polynomial $f_j$.  Denoting $S$ as $i_1 < i_2 < \ldots < i_M$, we can rewrite the above as
\[  \beta \left [ \left ( \sum_{i=0}^{k-1} \epsilon_{i_1, i} c_i \right ) \theta^{i_1}  +  \left ( \sum_{i=0}^{k-1} \epsilon_{i_2, i} c_i \right ) \theta^{i_2} + \ldots + \left ( \sum_{i=0}^{k-1} \epsilon_{i_M, i} c_i \right ) \theta^{i_M} \right ].  \]
Distribute $\beta$ to each summand, and expand it uniquely for each power of $\theta$ to get
\begin{align}\label{hugesums} 
\begin{split}& =  \sum_{j=0}^{k-1} \sum_{i=1}^{\ell^t} \epsilon_{i_1, j} c_j ( y_{1,i} - y_{1,\ell^{t} + i}) \theta^{i_1} 
 + \ldots +   \sum_{j=0}^{k-1} \sum_{i=1}^{\ell^t} \epsilon_{i_M, j} c_j ( y_{M,i} - y_{M,\ell^{t} + i}) \theta^{i_M}
\end{split}
\end{align}
Since our choices of $\theta$ and $y_{i,j}$ satisfy \eqref{amazing}, we have that each element in this sum is in $\pm \alpha * A$.  Hence, we have that for $\ell_1, \ell_2$ large enough,
\[ | \ell_1 ( \alpha * A) - \ell_2 (\alpha * A)| = |\ell_1  A - \ell_2 A| \ge \prod_{i=0}^{k-1} |C_i| \ge \left \lfloor \frac{|C|}{k} \right \rfloor^{k-1} \gg_k |A|^{\frac{k-1}{4}}. \] 
Recall that $\ell = k^{8}$, $M \le k^3$, and $t = \lceil \log_2{M} \rceil \le  \log_2{2k^3}  \le \log_e{3k^5}$.  So, we have $M \cdot 2\ell^{t}$ nonzero terms in $\sigma \in \Sigma$.  We bound this as
\[ M \cdot 2\ell^{t} \le 2k^3 k^{8 \log_e{3k^5}} = 2 k^{100 \log{k}} \]  
So, choosing $k:= e^{ \sqrt{ \frac{1}{100} \log{ \frac{h}{2}}  } }$ proves our theorem:
\[ |hA|  \ge \sqrt{|hA - hA|} \ge |A|^{\Omega(e^{\sqrt{\frac{1}{100}\log{h}}})} \] 
\end{proof}
\subsection{The Iterative Case}
We are now able to prove \cref{mainthm}.
\begin{proof}[Proof of \cref{mainthm}]
We iteratively apply \cref{mainprop} in the following algorithm.

\noindent \textbf{Step 0: } Let $k$ and $\ell$ be functions of $h$ as specified in the statement of \cref{mainprop}, and let $0 < \epsilon < \epsilon'$ where $\epsilon'$ is some unspecified function of $h$ taken to be sufficiently small.  Let $\ell_0 := \ell$, $A_0 := A$, and $\epsilon_0 := \epsilon$.  Since $|A_0 . A_0| \le |A_0|^{1+\epsilon_0}$, we may apply \cref{mainprop} to $A_0$.  If $|h A_0| \ge |A_0|^{\Omega(k)}$, then we are done.  Else, there exists a $j \in \{2, \ldots , \log{8k^5} \}$ and an $A_0' \subseteq A_0$ such that
\[ |(\ell^{j} + \ell^{j-1})A_0| \gg_{h} |A_0|^{\frac{1}{22k^{6}}} |\ell^j A_0'| \text{ and } |A_0'| \ge |A_0|^{1-c\epsilon} \]
where $c$ is a constant depending on $h$.  Let $A_1 := A_0'$ and continue to Step 1.

For $j=1, \ldots , \frac{1}{2} \ell$, we do the following.

\noindent \textbf{Step j: } Let $A_j$ be as specified in the previous step.  Since 
\[ |A_j . A_j| \le |A_{j-1}.A_{j-1}| \le  |A_{j-1}|^{1+\epsilon_{j-1}} \le |A_j|^{\frac{1+\epsilon_{j-1}}{1-c\epsilon_{j-1}}} \le |A_j|^{1+2c \epsilon_{j-1}} \]
where we assumed $\epsilon_{j-1}$ is sufficiently small in the last inequality.  Let $\epsilon_j := 2c \epsilon_{j-1}$.  Let $\ell_j :=  \ell - j$.  This determines $h_j$ and $k_j$ as
\[ h_j = e^{\frac{25}{36} (\log{(\ell-j)})^2} \text{ ; } k_j = (k^{8}-j)^{1/8}. \]
Applying \cref{mainprop} to $A_j$ with and $h_j$, we get that either
\[ |hA| \ge |h_j A_j | \ge |A_j|^{\Omega(k_j)} \ge |A_j|^{\Omega((k^{8}-j)^{1/8})} = |A_j|^{\Omega(k)} \ge |A|^{(1-(2c)^j\epsilon)\cdot \Omega(k)} = |A|^{\Omega(k)} \]
which proves the theorem for $\epsilon$ sufficiently small -- so we exit the algorithm.  Or, there exists an $A_j' \subseteq A_j$ of size $|A_j'| \ge |A_j|^{1-c\epsilon_j}$ and a $t_j \in \{2, \ldots , \log{8k^5} \}$ such that
\[ |(\ell_j^{t_j} + \ell_j^{t_j - 1}) A_j| \ge  |A_j|^{\frac{1}{22k_{j}^{6}}} |\ell_j^{t_j} A_j'| \ge  |A_j|^{\frac{1}{22k^{6}}}|\ell_j^{t_j} A_j'| \ge n^{\frac{1}{23k^{6}}} |\ell_j^{t_j} A_j'| \]
where we used the fact that $\epsilon$ is sufficiently small and $n$ is sufficiently large depending on $h$ in the last inequality.  Letting $A_{j+1} := A_j'$ we continue to Step $j+1$.

\noindent \textbf{Analysis of Algorithm: } Since $\ell_{j} = \ell - j$, and we perform at most $\ell/2$ steps, $\ell_j \ge \ell/2$. Assume the algorithm runs and finishes Step $\ell/2$.  Each step in the algorithm produces a $t_j \in \{2, \ldots , \log{8k^5}\}$.  By averaging, there is some integer $s \in \{2, \ldots , \log{8k^5}\}$ that appears in the algorithm at least $\frac{\ell}{2\log{8k^5}}$ times.  Denote  $j_1, \ldots , j_q$ as the steps in which $s$ is chosen.  It is easy to verify that by the definition of $\ell_j$,
\[ \ell_j^2 + \ell_j \le \ell_{j-1}^2, \]
and so we must also have that
\[ \ell_j^{s} + \ell_j^{s-1} \le \ell_{j-1}^2 \cdot \ell_j^{s-2} \le \ell_{j-1}^{s}. \]
So,
\begin{align*}  |(\ell_{j_1}^{s} + \ell_{j_1}^{s-1})A_{j_1}| & \ge n^{\frac{1}{23k^{6}}} |\ell_{j_1}^{s}A_{j_1}'| \ge
\\  & \ge n^{\frac{1}{23k^6}} |(\ell_{j_2}^{s} + \ell_{j_2}^{s-1})A_{j_2}| \ge n^{\frac{2}{23k^{6}}} |\ell_{j_2}^{s}A_{j_2}'| \ge
\\ & \vdots
\\ & \ge n^{\frac{q}{23k^{6}}} |(\ell_{j_q}^{s} + \ell_{j_q}^{s-1})A_{j_q}| \ge n^{\frac{\ell}{2\log{8k^5}} \cdot \frac{1}{23k^{6}}} = n^{\Omega(k)}
\end{align*}
where we used the fact that $q\ge \frac{\ell}{2\log{8k^5}}$ and $\ell = k^{8}$ in the last inequality.
Since
\[ \ell_{j_1}^s \le \ell^{\log{8k^5}} \le k^{8 (\log{8k^5})} \le  k^{100\log{k}} = h \]
we have that
\[ |hA|^2 \ge |(\ell_{j_1}^{s} + \ell_{j_1}^{s-1})A_{j_1}| \ge n^{\Omega(k)} \]
proving our theorem.
\end{proof}
The authors thank Jacob Fox for simplifying the original statement and proof of \cref{2caselemma}.

\pagebreak
\end{document}